\documentclass[11pt]{amsart}
\usepackage[sorted-cites]{amsrefs}
\usepackage[matrix,arrow,tips,curve]{xy}
\usepackage[english]{babel}
\usepackage{amsmath}
\usepackage{amssymb}
\numberwithin{equation}{section}
\usepackage{ mathrsfs }
\theoremstyle{plain}
\usepackage{enumerate}
\usepackage{graphicx}
\usepackage{hyperref}
\usepackage{color}
\DeclareMathAlphabet{\pazocal}{OMS}{zplm}{m}{n}

\usepackage{xcolor}

\newtheorem{theorem}{Theorem} [section]
\newtheorem{lemma}{Lemma}[section]
\newtheorem{proposition}{Proposition}[section]

\newtheorem{cor}{Corollary}[section]
\newtheorem{definition}{Definition}
\newtheorem{example}{Example}
\theoremstyle{remark}
\newtheorem{remark}{Remark}[section]

\hypersetup{pdfpagemode=UseNone} 
\hypersetup{pdfstartview=FitH}

\begin{document}

\title[The Wasserstein distance for Ricci shrinkers]{The Wasserstein distance for Ricci shrinkers}



\thanks{F. Conrado was partially supported by CNPq/Brazil [Grant: 151430/2022-5].}
\thanks{D. Zhou was partially supported by CNPq/Brazil [Grant: 305364/2019-7, 403344/2021-2] and FAPERJ/ Brazil [Grant: E-26/200.386/2023, E-26/210.120/2020].}

\address{Departamento de Matem\'atica, Universidade Federal de Sergipe, Jardim Rosa Elze, S\~ao Cristóvão, SE  49100-000, Brazil}
\author[Franciele Conrado]{Franciele Conrado}

\email{franciele@mat.ufs.br}
\author[Detang Zhou]{Detang Zhou}
\address{Departamento de Geometria, Instituto de Matem\'atica e Estat\'\i stica, Universidade Federal Fluminense, S\~ao Domingos,
Niter\'oi, RJ 24210-201, Brazil}
\email{zhoud@id.uff.br}

\newcommand{\M}{\mathcal M}

\begin{abstract} Let $(M^n,g,f)$ be a Ricci shrinker such that $\textrm{Ric}_f=\frac{1}{2}g$ and the measure induced by the weighted volume element $(4\pi)^{-\frac{n}{2}}e^{-f}dv_{g}$ is a probability measure. Given a point $p\in M$, we consider two probability measures defined in the tangent space $T_pM$, namely the Gaussian measure $\gamma$ and the measure $\overline{\nu}$ induced by the exponential map of $M$ to $p$. In this paper, we prove a result that provides an upper estimate for the Wasserstein distance with respect to the Euclidean metric $g_0$ between the measures $\overline{\nu}$ and $\gamma$, and which also elucidates the rigidity implications resulting from this estimate.
\end{abstract}

\maketitle
\section{Introduction}

\medskip

A \textit{complete smooth metric measure space} $(M^n, g, f)$ refers to a complete $n$-dimensional Riemannian manifold $(M^n, g)$ accompanied by a smooth function $f$ defined on $M$. The \textit{Bakry-Émery Ricci curvature} of a complete smooth metric measure space $(M^n, g, f)$ is defined as:
$$\textrm{Ric}_f=\textrm{Ric}_g+\nabla^2f$$
\noindent where $\nabla^2 f$ represents the Hessian of $f$ and $\textrm{Ric}_g$ represents the Ricci curvature of $(M, g)$. A complete smooth metric measure space $(M^n, g, f)$ is referred to as a \textit{shrinking gradient Ricci soliton} or simply a \textit{Ricci shrinker} if:
$$\textrm{Ric}_f=\frac{1}{2\tau} g$$
\noindent for some positive constant $\tau$. In this case, the function $f$ is referred to as a \textit{potential function} of the Ricci shrinker.

\medskip

It is worth noting that any Einstein manifold with positive scalar curvature, combined with a constant function, is considered a Ricci shrinker. An important example of a Ricci shrinker is the \textit{Gaussian shrinker} $(\mathbb{R}^n, g_e, \frac{|x|^2}{4})$, where $g_e$ represents the Euclidean metric and the potential function is given by $f(x) = \frac{|x|^2}{4}$ for $x \in \mathbb{R}^n$. Another example of a Ricci shrinker is the \textit{cylinder Ricci shrinker} $(\mathbb{S}^{n-k}(r) \times \mathbb{R}^k, g_c, f_c)$ with the product metric $g_c$ and potential function $f_c(x,y) = \frac{|y|^2}{4}$, where $(x,y) \in \mathbb{S}^{n-k}(r) \times \mathbb{R}^k$. Here, $\mathbb{S}^{n-k}(r)$ denotes the $(n-k)$-dimensional round sphere of radius $r = \sqrt{2(n-k-1)}$, with $k \geq 1$ and $n-k \geq 2$.

\medskip

The study of Ricci shrinkers lies at the intersection of critical metrics and geometric flows, and it has been extensively studied, yielding numerous significant results. In dimensions 2 and 3, all Ricci shrinkers have been completely classified. Hamilton \cite{hamilton} demonstrated that any 2-dimensional Ricci shrinker is either isometric to $\mathbb{R}^2$ or a finite quotient of $\mathbb{S}^2$. Additionally, Cao, Chen, and Zhu \cite{ccz}, Naber \cite{naber}, and Ni and Wallach \cite{ni} established that any 3-dimensional Ricci shrinker is isometric to a finite quotient of $\mathbb{R}^3$, $\mathbb{S}^3$, or $\mathbb{S}^2 \times \mathbb{R}$.

\medskip

In higher dimensions, the classification of gradient shrinking Ricci solitons becomes more complex, especially due to the Weyl tensor's involvement. In the four dimensional case, some classification results have been obtained under specific conditions. Locally conformally flat solitons are finite quotients of $\Bbb{S}^4$, $\Bbb{R}^4$, or $\Bbb{S}^{3}\times\Bbb{R}$. Half-conformally flat solitons are finite quotients of $\Bbb{S}^4$, $\Bbb{CP}^2$, $\Bbb{R}^4$, or $\Bbb{S}^{3}\times\Bbb{R}$. Bach-flat solitons are either Einstein or finite quotients of $\Bbb{R}^4$ or $\Bbb{S}^{3}\times\Bbb{R}$. Under certain vanishing conditions on the Weyl tensor, solitons can be classified as Einstein or finite quotients of specific spaces. Notably, the classification of four-dimensional Ricci shrinkers with constant scalar curvature has recently been fully achieved through the works of Peterson and Wylie \cite{1dim4}, Fernández-López and García-Río \cite{2dim4}, and Cheng and Zhou \cite{3dim4}. These studies have established that any 4-dimensional Ricci shrinker with constant scalar curvature is isometric to one of the following: an Einstein manifold, a finite quotient of $\mathbb{R}^4$, $\mathbb{S}^2\times \mathbb{R}^2$, or $\mathbb{S}^3\times \mathbb{R}$.

\medskip

More partial classification results regarding Ricci shrinkers in higher dimensions can be found in the works mentioned, including \cite{CRZ-dim4}, \cite{cda1}, \cite{cda2}, \cite{cda3}, \cite{cda4}, \cite{cda5}, \cite{ni}, \cite{cda7}, and \cite{cda8}. These contributions have shed light on the understanding of Ricci shrinkers beyond dimension four, but a complete classification is still an ongoing research endeavor.

\medskip

The question of rigidity in geometry pertains to whether a particular geometric object is uniquely determined within its moduli space. Rigidity results aim to describe how ``rigid" a given geometric structure is, i.e., how few variations or deformations it allows. In the context of Ricci shrinkers, there are some interesting rigidity results. In this paper, we are only interested in rigidity results for non-compact Ricci shrinkers. For the compact case, see some important works \cite{comp1}, \cite{comp2}, \cite{comp3}, \cite{comp4} and \cite{comp5}.

\medskip

Let us list several types of different and related results. Firstly Yokota, as referenced in \cite{y1} and \cite{y2}, used techniques for estimating the reduced volume to prove a result regarding Ricci shrinkers with entropy close to zero. Yokota demonstrated that any Ricci shrinker $(M^n,g,f)$ with entropy sufficiently close to zero is actually the Gaussian shrinker $(\mathbb{R}^n, g_e, \frac{|x|^2}{4})$. In particular, this result proves the Carrillo–Ni conjecture, which states that the Gaussian shrinker is the only Ricci shrinker with zero entropy (see \cite{carrillo}, Remark 7.2).

\medskip

Secondly, Li and Wang \cite{lw} proved a rigidity result for Ricci shrinkers close to the cylinder Ricci shrinker $(\mathbb{S}^{n-1}(r)\times \mathbb{R},g_c,f_c)$. They showed that if a Ricci shrinker is sufficiently close to the cylinder shrinker in the pointed-Gromov-Hausdorff topology, then it must be the cylinder shrinker itself. This result implies that the cylinder Ricci shrinker is isolated in the moduli space of all Ricci shrinkers with the pointed-Gromov-Hausdorff topology.

\medskip

Lastly, Cheng and Zhou \cite{XCDZ} proved that the spectrum of the drift Laplacian on a Ricci shrinker is discrete and there is a sharp lower bound for the first nonzero eigenvalue, and equality is achieved if and only if the shrinker splits off a line.

\medskip

Colding and Minicozzi \cite{CM} showed a result known as propagation of almost splitting. It states that if a Ricci shrinker almost splits on one scale, it also almost splits on larger scales. More precisely, if a Ricci shrinker is sufficiently close to the cylinder Ricci shrinker $(\mathbb{S}^{n-k}(r)\times \mathbb{R}^k,g_c,f_c)$ within a large compact set, then it must be the cylinder shrinker itself.

\medskip

To summarize, these rigidity results for non-compact Ricci shrinkers include the isolation of the cylinder shrinker in the moduli space, the confirmation of the Carrillo-Ni conjecture by Yokota's work, the discreteness of the spectrum of the drift Laplacian, and the propagation of almost splitting property. The references mentioned, such as \cite{comp1}, \cite{comp2}, \cite{comp3}, \cite{comp4}, \cite{comp5}, \cite{lw}, \cite{y1}, \cite{y2}, \cite{carrillo}, \cite{XCDZ}  and \cite{CM}, likely correspond to specific papers or sources that contain the complete details and proofs of the results discussed.

\medskip

Let $(M^n,g,f)$ be a Ricci shrinker such that $\int_M(4\pi)^{-\frac{n}{2}}e^{-f}dv_{g}=1$ and $\textrm{Ric}_f=\frac{1}{2}g$, where $dv_g$ denotes the volume element of $(M^n,g)$. Given a point $p\in M$, it is well-known that there is a domain $\Omega\subset T_pM$ such that $\left.\textrm{exp}_p\right|_{\Omega}:\Omega\rightarrow M\setminus \mathcal{C}(p)$ is a diffeomorphism, where $\mathcal{C}(p)$ denote the cut-locus of $p$. 

\medskip

\begin{definition} The measure $\overline{\nu}$ on $T_pM$ induced by the exponential map of $M$ to $p$ is defined by
$$\overline{\nu}=(4\pi)^{-\frac{n}{2}}\chi_{\Omega}e^{-\overline{f}}dv_{\overline{g}}$$
\noindent where  $\overline{g}$ and $\overline{f}$ denote the pullback of $g$ and $f$ by the map $\left.\textrm{exp}_p\right|_{\Omega}:\Omega\rightarrow M\setminus \mathcal{C}(p)$, respectively. Here, $\chi_{\Omega}$ denotes the characteristic function of $\Omega$.
\end{definition}

\medskip

In this paper, we will consider two probability measures defined in the tangent space $T_pM$, namely the Gaussian measure $\gamma$ and the measure $\overline{\nu}$ induced by the exponential map of $M$ to $p$. Both measures are probability measures with finite second moments with respect to the Euclidean metric $g_0$. Furthermore, the measure $\overline{\nu}$ is absolutely continuous with respect to $\gamma$.

\medskip

The Wasserstein distance quantifies the dissimilarity or discrepancy between probability measures, taking into account the underlying geometry of the space (see Section \ref{sec2} for more details). The main result of this paper is a theorem that provides an upper estimate for the Wasserstein distance $W_{g_0}$ between the measures $\overline{\nu}$ and $\gamma$. It not only provides an upper bound for the Wasserstein distance between $\overline{\nu}$ and $\gamma$, but also elucidates the rigidity implications resulting from this estimate.

\begin{theorem}\label{teopr} Let $(M^n,g,f)$ be a Ricci shrinker such that $\textrm{Ric}_f=\frac{1}{2}g$ and  $\int_M(4\pi)^{-\frac{n}{2}}e^{-f}dv_{g}=1$. Then, for any point $p\in M$, the Wasserstein distance between the probability measure $\overline{\nu}$ on $T_pM$ and the Gaussian measure $\gamma$ on $T_pM$ satisfies
\begin{equation}\label{eqp12} \frac{1}{4}W_{g_0}^2(\overline{\nu},\gamma) \leq \alpha e^{f(p)-\mu_g}+f(p)-\mu_g\end{equation}

\noindent where $\mu_g$ denotes the entropy of the Ricci shrinker $(M^n,g,f)$, $\alpha$ is a nonnegative constant depending only on the dimension $n$, and nonnegative constants $a$ and $b$  such that
\begin{equation}\label{eqp1} f(x)\geq \frac{r^2(x)}{4}-ar(x)-b\end{equation}
\noindent for every $x\in M$. Here, $r(x)=d_g(p,x)$ is the geodesic distance function to the point $p$. Moreover, if the equality holds in (\ref{eqp12}) then $a=b=0$ and $(M^n,g,f)$ is $(\mathbb{R}^n,g_{e},\frac{|x-p|^2}{4})$, i.e., the Gaussian shrinker up to a translation.
\end{theorem}

\begin{remark} Cao and Zhou \cite{caozhou} proved that the potencial function a Ricci shrinker always admits Gaussian upper and lower estimates, in particular, this result guarantees the existence of nonnegative constants $a$ and $b$ satisfying  (\ref{eqp1}). Furthermore, the nonnegative constant  $\alpha=\alpha(n,a,b)$ has an explicit form that involves functions with expressions similar to the Gamma functions, and satisfies $\alpha=0$ if $a=b=0$ (see Section \ref{sec4} for the exact expression of $\alpha$).
\end{remark}

There are two key ingredients in the proof of Theorem \ref{teopr}. The first one is a generalization of the Talangrand inequality obtained by combining results from Otto and Villani \cite{ov}, and Bakry and Émery \cite{BE}. This inequality gives us a lower bound for the Wasserstein distance between a probability measure with finite second moments $\nu$ associated with a complete smooth metric measure spaces $(M^n,g,f)$ with $\textrm{Ric}_f\geq \rho g$ for some positive constant $\rho$, and any other probability measure with finite second moments that is absolutely continuous with respect to $\nu$. In Theorem \ref{cz} of this paper, we prove that for $(M^n,g,f)$, the assumption of $\textrm{Ric}_f\geq \rho g$ for some constant $\rho>0$ implies that a measure $\nu=e^{-f}dv_g$ has finite moments of all orders, in particular, has second finite moments. So combining the Theorems of Otto-Villani and Bakry-Émery, we have the following.

\begin{theorem}\label{ttov} Let $(M^n,g,f)$ be a complete smooth metric measure space such that $\nu=e^{-f}dv_g$ is a probability measure and $\textrm{Ric}_f\geq \rho g$, for some $\rho>0$. Then the Wasserstein distance $W_g$ between $\nu$ and any probability measure with finite second moments $\eta$ absolutely continuous with respect to $\nu$, satisfies:
$$W_g^2(\eta,\nu)\leq \frac{2}{\rho}\int_{M}\log \left(\frac{d \eta}{d\nu}\right)d\eta.$$
\end{theorem}

\medskip

We apply Theorem \ref{ttov} to the Ricci shrinker $(T_pM, g_0, \frac{|x|^2_{g_0}}{4})$, and we obtain 
\begin{equation}\label{eqp}W_{g_0}^2(\gamma,\overline{\nu})\leq 4\int_{M}\log \left(\frac{d \overline{\nu}}{d\gamma}\right)d\overline{\nu}.\end{equation}

\medskip

The second ingredient is the technique for the volume growth for geodesic balls developed by Cheng, Ribeiro e Zhou in \cite{exd}, which allows us to obtain an upper bound for the density function $\frac{d\overline{\nu}}{d\gamma}$. Combining this lower bound, with estimate (\ref{eqp1}) for the potential function $f$ and inequality (\ref{eqp}), we obtain estimate (\ref{eqp12}).

\medskip

On each Ricci shrinker, the potential function always attains its minimum value at some point (see \cite{caozhou}), which can be chosen as the base point. Then, as a consequence of Theorem \ref{teopr} we have the following result.

\begin{cor}\label{copr} Let $(M^n,g,f)$ be a Ricci shrinker such that $\textrm{Ric}_f=\frac{1}{2}g$ and $\int_M(4\pi)^{-\frac{n}{2}}e^{-f}dv_{g}=1$. Then, if $f$ attains its minimum value at point $p\in M$, the Wasserstein distance between the probability measure $\overline{\nu}$ on $T_pM$ and the Gaussian measure $\gamma$ on $T_pM$ satisfies
$$\frac{1}{4}W_{g_0}^2(\overline{\nu},\gamma) \leq \alpha e^{R_g(p)}+R_g(p)$$
\noindent where $R_g$ denotes the scalar curvature of $g$, $\alpha$ is a nonnegative constant depending only on the dimension $n$, and nonnegative constants $a$ and $b$  such that
$$f(x)\geq \frac{r^2(x)}{4}-ar(x)-b$$
\noindent for every $x\in M$. Moreover, if the equality holds then $a=b=0$ and $(M^n,g,f)$ is $(\mathbb{R}^n,g_{e},\frac{|x-p|^2}{4})$, i.e., the Gaussian shrinker up to a translation.
\end{cor}

An interesting observation is that the inequalities obtained in Theorem \ref{teopr} and Corollary \ref{copr} are optimal in the sense that the equalities hold for the Gaussian $(\mathbb{R}^n,g_{e},\frac{|x|^2}{4})$ up to a suitable translation and conversely, these equalities imply that the Ricci shrinker must be the Gaussian shrinker $(\mathbb{R}^n,g_{e},\frac{|x|^2}{4})$ up to a translation.

\medskip

The article is organized as follows. In Section \ref{sec2}, we fix some notations and recall key definitions and necessary results in our arguments. Theorem \ref{ttov} is found in Section \ref{sec3}, where we prove that the measure associated with a complete smooth metric measure space with a positive lower bound under Bakry-Émery Ricci curvature has finite moments of all orders. Finally, in Section \ref{sec4}, we present the proof of the main Theorem \ref{teopr}  and some of its consequences.

\medskip

\noindent {\bf Acknowledgment.} The first author would like to thank the Conselho Nacional de Desenvolvimento Científico e Tecnológico (CNPq) for the financial support and the Instituto de Matemática e Estatística at the Universidade Federal Fluminense for the hospitality during her postdoctoral studies.

\section{Preliminaries}\label{sec2}

In this section, we fix some notations and recall fundamental definitions and necessary results concerning Ricci shrinkers and the Wasserstein distance for a better understanding of the subsequent sections.

\subsection{Ricci shrinkers}

Sometimes we refer to a Ricci shrinker  $(M^n, g, f)$ such that $\textrm{Ric}_f =\frac{1}{2\tau}g$ simply as a {\it $\tau$-Ricci shrinker}. It is well-known that any Ricci shrinker $(M^n,g,f)$ has nonnegative scalar curvature $R_g\geq 0$ (see \cite{chen}). Furthermore, by the maximum principle, we have that  $R_g>0$ unless $R_g\equiv 0$, and $(M^n,g)$ is isometric to $(\mathbb{R}^n,g_e)$ (see \cite{prs}, Theorem 3).

\begin{example}\label{ex2} Let $(M^n,g)$ be a complete Riemannian manifold and fix point $p\in M$. The triple $(T_pM,g_0,f_0)$, where $g_0$ is the Euclidean metric on $T_pM$ and $f_0(x)=\frac{|x|_{g_0}^2}{4}$, is a $1$-Ricci shrinker. The measure $\gamma=(4\pi)^{-\frac{n}{2}}e^{-f_0}dv_{g_0}$ on $T_pM$  is a probability measure and is called the {\it Gaussian measure}.
\end{example}

\medskip

Let's recall some useful facts about full Ricci shrinkers.

\medskip

\begin{theorem}[Hamilton, \cite{hamilton}]\label{H} Let $(M^n,g,f)$ be a $\tau$-Ricci shrinker. Then there is a constant $C$ such that
$$\tau(R_g+|\nabla f|^2)=f+C.$$
\end{theorem}

\medskip

The following theorem gives us an estimate for the potential function of a noncompact Ricci shrinker.

\medskip

\begin{theorem}[Cao and Zhou, \cite{caozhou}] Let $(M^n,g,f)$ be a noncompact $1$-Ricci shrinker with $f=|\nabla f|^2+R_g$. Fix a point $p\in M$ and consider the distance function $r(x)=d_g(x,p)$, $x\in M$. Then
$$\frac{1}{4}(r(x)-c)^2\leq f(x)\leq \frac{1}{4}(r(x)+c)^2$$
\noindent for every $x\in M$ with $r(x)$ sufficiently large, where $c$ is a positive constant depending only on $n$ and the geometry of $g$ on the unit ball $B_1(p)$.
\end{theorem}

\begin{remark} It follows from previous theorem that the potential function of a Ricci shrinker attains a global minimum.
\end{remark}

\medskip

Let $(M^n,g,f)$ be a $\tau$-Ricci shrinker. Its {\it $f$-volume} is defined by
$$V_f(M)=\int_M (4\pi\tau)^{-\frac{n}{2}}e^{-f}dv_g.$$

It was proved by Cao and Zhou in \cite{caozhou} that $V_f(M)$ is finite. Then, the {\it entropy} of the Ricci shrinker is given by
$$\mu_{g}=\log V_f(M)-C,$$

\noindent where $C\in \mathbb{R}$ is the constant given in Theorem \ref{H}. Note that the entropy of the Gaussian shrinker is equal to zero. In \cite{y1} and \cite{y2}, Yokota proved that $\mu_g\leq 0$ and a gap theorem showing how the entropy is involved with the global geometry.

\begin{theorem}[Yokota, \cite{y1} and \cite{y2}] There is a positive number $\epsilon=\epsilon(n) > 0$ depending only on the dimension $n$ with the following property: Any Ricci shrinker $(M^n,g,f)$ with $\mu_g\geq -\epsilon$ is the Gaussian shrinker $(\mathbb{R}^n,g_{e},\frac{|x|^2}{4})$.
\end{theorem}

The above theorem states that the entropy of a nonflat Ricci shrinker cannot be too close to zero. Observe that as a consequence of the above theorem we have the Gaussian shrinker is the only Ricci shrinker with entropy equal to zero.

\begin{remark}\label{rentropy} Let $(M^n,g,f)$ be a $\tau$-Ricci shrinker. Then, the quantity $\mu_g$ satisfies
$$\tau(R_g+|\nabla f|^2)=f+\log V_f(M)-\mu_g.$$

\medskip

In particular, we have $V_f(M)=1$ if and only if
$$\tau(R_g+|\nabla f|^2)=f-\mu_g.$$
\end{remark}

\begin{example}Consider the cylinder Ricci shrinker $(\mathbb{S}^{n-k}(r) \times \mathbb{R}^k, g_c, f_c)$. In this case, we have that
$$\mu_{g_c}=\log \omega_{n-k}-\frac{(n-k)}{2}\left[\log\left(\frac{2\pi}{n-k-1}\right)+1\right],$$
\noindent where $\omega_{n-k}$ denotes the area of the unit $(n-k)$-dimensional sphere.
\end{example}

\medskip

\subsection{The Wasserstein distance} 

\medskip

Let $(M^n,g)$, $n\geq 2$, be a complete Riemannian manifold  with the distance $d_g$ induced by Riemanniann metric $g$. Recall that a measure $\nu$ on $M$ has {\it finite $k$-moments} with respect to $g$ ($k>0$), if 
$$\int_M d_g^k(p,.)d\nu< +\infty$$
\noindent for some (and hence any) $p\in M$. Denote by $\mathcal{P}(M)$ the set of all Borel probability measures on $M$ and $\mathcal{P}_k(M,g)$ the set of all Borel probability measures on $M$ with finite $k$-moments with respect to $g$.

\begin{remark} Let $B\subset M$ be a Borel set. Observe that if $\nu\in \mathcal{P}_k(M,g)$ then $\left.\nu\right|_B\in \mathcal{P}_k(M,g)$. Here $\left.\nu\right|_B$ is the restriction of $\nu$ to $B$, i.e., $$d\left.\nu\right|_B=\displaystyle\frac{\chi_B}{\nu(B)}d\nu.$$
\end{remark}

Given $\nu, \eta\in \mathcal{P}(M)$, we say that $\pi\in \mathcal{P}(M\times M)$ is a {\it coupling} of $(\nu,\eta)$ if $$\pi(A\times M)=\nu(A) \ \ \ \text{and} \ \ \ \pi(M\times A)=\eta(A)$$
\noindent for every Borel set $A\subset M$. We denote by $\Pi(\nu,\eta)$ the set of all coupling of $(\nu,\eta)$. For any $\eta,\nu\in\mathcal{P}(M)$ we have that $\Pi(\nu,\eta)\not=\emptyset$ because the product measure $\nu\times \eta$ is a coupling of $(\nu,\eta)$.

\medskip

Recall that the {\it Wasserstein distance} between two measures $\eta,\nu\in\mathcal{P}_2(M,g)$ is given by

\begin{equation}\label{wdis}W_g(\eta,\nu)=\left(\inf_{\pi\in \Pi(\eta,\nu)}\int_{M\times M}d_g^2(x,y)d\pi(x,y)\right)^{\frac{1}{2}}.\end{equation}

A coupling $\pi\in \Pi(\eta,\nu)$ is said to be {\it optimal} if it attains the infimum in (\ref{wdis}). For any $\eta,\nu\in\mathcal{P}_2(M,g)$ we have that $W_g(\eta,\nu)$ is finite and an optimal coupling always exists (see \cite{villani}, Chapter 4). Moreover, $(\mathcal{P}_2(M,g),W_g)$ is a complete and separable metric space (see \cite{villani}, Chapter 6).

\medskip

\begin{example}
Fix a point $p\in M$ and consider $\delta_p$ the Dirac measure at $p$. For any $\nu\in\mathcal{P}_2(M,g)$, the product measure $\delta_p\times \nu$ is the only coupling of $(\delta_p,\nu)$, which therefore is optimal, and then
$$W_g(\delta_p,\nu)=\left(\int_Md_g^2(p,.)d\nu\right)^{\frac{1}{2}}.$$

In particular, for any $q\in M$ we have
$$W_g(\delta_p,\delta_q)=d_g(p,q).$$
\end{example}

\medskip

\begin{remark} Observe that there is a natural isometric immersion of $(M,d_g)$ into $(\mathcal{P}_2(M,g),W_g)$, namely the map $x\mapsto \delta_x$.
\end{remark}

\medskip

Consider now a complete smooth metric measure space $(M^n,g,f)$ such that $\nu=e^{-f}dv_g\in \mathcal{P}(M)$. Recall that if $\eta\in \mathcal{P}(M)$ is absolutely continuous with respect to $\nu$, we define its {\it relative entropy} with respect to $\nu$ by
$$H(\eta|\nu)=\int_{M}\log \left(\frac{d\eta}{d\nu}\right)d\eta$$

\noindent and, we define its {\it relative Fisher information} with respect to $\nu$ by
$$I(\eta|\nu)=\int_{M}\left|\nabla \log\left(\frac{d\eta}{d\nu}\right)\right|^2d\eta.$$

We say that $\nu$ satisfies a logarithmic Sobolev inequality with constant $\rho>0$, or satisfies LSI($\rho$) for short, if for all measure $\eta\in\mathcal{P}(M)$ absolutely continuous with respect to $\nu$, we have
$$H(\eta | \nu)\leq \frac{1}{2\rho}I(\eta | \nu).$$

If $\nu\in \mathcal{P}_2(M,g)$, we say that $\nu$ satisfies a Talagrand inequality with constant $\rho>0$, or satisfies T($\rho$) for short, if for all measure $\eta \in \mathcal{P}_2(M,g)$ absolutely continuous with respect to $\nu$, we have
$$W_g^2(\eta,\nu)\leq \frac{2}{\rho}H(\eta | \nu).$$

Talagrand proved that if $(M^n,g,f)$ is the Ricci shrinker $(\mathbb{R}^n,g_e, \frac{|x|^2}{2}+\frac{n}{2}\log(2\pi))$ então $\nu$ satisfies a Talagrand inequality with constant $\rho=1$. In \cite{ov}, Otto and Villani generalized this Talagrand result to a very wide class of complete smooth metric measure space: namely, all complete smooth metric measure spaces $(M^n,g,f)$ with $\nu\in \mathcal{P}_2(M,g)$ satisfying the logarithmic inequality of Sobolev for some positive constant and lower bounded Bakry-Émery Ricci curvature. More precisely, they proved the following result.

\begin{theorem}[Otto and Villani, \cite{ov}]\label{ov} Let $(M^n,g,f)$ be a complete smooth metric measure space such that $\nu=e^{-f}dv_g$ is a probability measure with finite second moments and $Ric_f\geq cg$, for some $c\in\mathbb{R}$. If $\nu$ satisfies  LSI($\rho$) for some $\rho>0$, then it also satisfies T($\rho$).
\end{theorem}

We now recall a well-known sufficient condition for the measure $\nu$ to satisfy the logarithmic Sobolev inequality.

\begin{theorem}[Bakry and Émery, \cite{BE}]\label{be} Let $(M^n,g,f)$ be a complete smooth metric measure space such that $\nu=e^{-f}dv_g$ is a probability measure and $Ric_f\geq \rho g$, for some $\rho>0$. Then $\nu$ satisfies  LSI($\rho$).
\end{theorem}

As a consequence of theorems \ref{ov} and \ref{be} we have the following result.

\begin{theorem}\label{ovbe}Let $(M^n,g,f)$ be a complete smooth metric measure space such that $\nu=e^{-f}dv_g$ is a probability measure with finite second moments and $Ric_f\geq \rho g$, for some $\rho>0$. Then $\nu$ satisfies T($\rho$), i.e., for any probability measure with finite second moments $\eta$ absolutely continuous with respect to $\nu$, we have 
\begin{equation}\label{ww}W_g^2(\eta,\nu)\leq \frac{2}{\rho}\int_{M}\log \left(\frac{d \eta}{d\nu}\right)d\eta.\end{equation}
\end{theorem}

\medskip

For more details on the Wasserstein distance see \cite{villani}.

\medskip

\section{Finiteness of $k$-moments}\label{sec3}

\medskip

In this section, we consider a complete smooth metric measure space $(M^n,g,f)$ with $\textrm{Ric}_f\geq \rho g$, for some $\rho\in \mathbb{R}$. We will use the technique for the volume growth for geodesic balls developed in \cite{exd} and the Munteanu-Wang's estimate for the potential function $f$ obtained in \cite{munteanu}. As a special case we have the following result.

\begin{theorem}\label{cz} Let $(M^n,g,f)$ be a complete smooth metric measure space such that $\textrm{Ric}_f\geq \rho g$, for some constant $\rho>0$. Then $\nu=e^{-f}dv_g$ has finite $k$-moments with respect to $g$, for every $k>0$.
\end{theorem}

\begin{remark}\label{rfm} It follows from Theorem \ref{cz} that the hypothesis of finiteness of the second moments under measure $\nu$ on Theorem \ref{ov} is superfluous for (\ref{ww}). So combining the Theorems \ref{ovbe} and \ref{cz}, we have the following result. 
\end{remark}

\begin{theorem}[Theorem \ref{ttov}]\label{ovbecz}Let $(M^n,g,f)$ be a complete smooth metric measure space such that $\nu=e^{-f}dv_g$ is a probability measure and $\textrm{Ric}_f\geq \rho g$, for some $\rho>0$. Then $\nu$ has finite second moments, and for any probability measure with finite second moments $\eta$ absolutely continuous with respect to $\nu$, we have 
$$W_g^2(\eta,\nu)\leq \frac{2}{\rho}\int_{M}\log \left(\frac{d \eta}{d\nu}\right)d\eta.$$
\end{theorem}

Theorem \ref{cz} is a consequence of the following general result.

\begin{proposition}\label{t1}
Let $(M^n,g,f)$ be a complete smooth metric measure space with $\textrm{Ric}_f\geq \rho g$, for some $\rho\in \mathbb{R}$. Fix a point $p\in M$ and consider the distance function $r(x)=d_g(x,p)$, $x\in M$. Then, there exist constants $r_0, A, B, C>0$ such that
$$\int_{B_R(p)}(\varphi\circ r) e^{-f}dv_g \leq A\varphi(r_0)+B\int_{r_0}^R \varphi(r) r^{n-1}e^{-\frac{\rho r^2}{2}+Cr}dr$$
\noindent for every increasing function $\varphi$ and $R\geq r_0$.
\end{proposition}
\begin{proof} It is natural to use the exponential map at $p$ as many authors have used to write, in terms of polar normal coordinates at $p$, the volume element as
$$dv_g=\mathcal{A}(r,\theta)dr d\theta$$
\noindent where $d\theta$  is the area element of the unit $(n-1)$-dimensional sphere $\mathbb{S}^{n-1}$ (see for example \cite{li}, page 10). Denote by $S_r(p)$ the geodesic sphere on $(M,g)$ of radius $r$ centered at $p$. The area element of $S_r(p)$ is given by $\mathcal{A}(r,\theta)d\theta$. It follows from first variation of the area that if $x=(r,\theta)$ is not in the cut-locus of $p$ then
\begin{equation}\label{x01}\frac{d}{dr}\log(\mathcal{A}(r,\theta))=H(r,\theta),\end{equation}
\noindent where $H(r,\theta)$ denote the mean curvature of $S_r(p)$ at the point $x=(r,\theta)$ with inward pointing normal
vector. By Bochner formula and the Schwarz inequality, we have
$$\frac{d}{dr}H(r,\theta)+\frac{H^2(r,\theta)}{n-1}+\textrm{Ric}_g\left(\frac{\partial}{\partial r},\frac{\partial}{\partial r}\right)(r,\theta)\leq 0.$$

Consequently, for $0<\epsilon\leq r$ we obtain
$$\int_{\epsilon}^rs^2\frac{d}{ds}H(s,\theta)ds+\frac{1}{n-1}\int_{\epsilon}^rs^2H^2(s,\theta)ds+\int_{\epsilon}^rs^2\textrm{Ric}_g\left(\frac{\partial}{\partial s},\frac{\partial}{\partial s}\right)ds\leq 0.$$

Using integration by parts on the first term and making $\epsilon \to 0$ we have
\begin{eqnarray*}
r^2H(r,\theta) &\leq &\int_{0}^r\left(2sH(s,\theta)-\frac{s^2H^2(s,\theta)}{n-1}\right)ds-\int_{0}^rs^2\textrm{Ric}_g\left(\frac{\partial}{\partial s},\frac{\partial}{\partial s}\right)ds\\
& = &-\frac{1}{n-1}\int_{0}^r[sH(s,\theta)-(n-1)]^2ds+\int_{0}^r\left[(n-1)-s^2\textrm{Ric}_g\left(\frac{\partial}{\partial s},\frac{\partial}{\partial s}\right)\right]ds\\
& \leq & \int_{0}^r\left[(n-1)-s^2\textrm{Ric}_g\left(\frac{\partial}{\partial s},\frac{\partial}{\partial s}\right)\right]ds.\\
\end{eqnarray*}

This implies that
\begin{equation}\label{x11}
H(r,\theta)\leq \frac{1}{r^2}\int_{0}^r\left[(n-1)-s^2\textrm{Ric}_g\left(\frac{\partial}{\partial s},\frac{\partial}{\partial s}\right)\right]ds.
\end{equation}

Note that, by (\ref{x01}) and (\ref{x11}) we  have
\begin{equation}\label{x21}
\frac{d}{dr}\log\left(\frac{\mathcal{A}(r,\theta)}{r^{n-1}}\right)\leq -\frac{1}{r^2}\int_{0}^rs^2\textrm{Ric}_g\left(\frac{\partial}{\partial s},\frac{\partial}{\partial s}\right)ds.
\end{equation}

Integrating (\ref{x21}) from $\epsilon>0$ to $r$ results in
$$\log\left(\frac{\mathcal{A}(r,\theta)}{r^{n-1}}\right)-\log\left(\frac{\mathcal{A}(\epsilon,\theta)}{\epsilon^{n-1}}\right)\leq -\int_{\epsilon}^r\left[\frac{1}{s^2}\int_{0}^st^2\textrm{Ric}_g\left(\frac{\partial}{\partial t},\frac{\partial}{\partial t}\right)dt\right]ds.$$

Since $\lim_{\epsilon\to 0}\log\left(\frac{\mathcal{A}(\epsilon,\theta)}{\epsilon^{n-1}}\right)=1$, using integration by parts on the last term and making $\epsilon \to 0$ we obtain
\begin{equation}\label{x31}
\log\left(\frac{\mathcal{A}(r,\theta)}{r^{n-1}}\right)\leq  \frac{1}{r}\int_{0}^rs^2\textrm{Ric}_g\left(\frac{\partial}{\partial s},\frac{\partial}{\partial s}\right)ds-\int_{0}^rs\textrm{Ric}_g\left(\frac{\partial}{\partial s},\frac{\partial}{\partial s}\right)ds.
\end{equation}

It follows from (\ref{x21}) and (\ref{x31}) that
$$\frac{d}{dr}\left(r\log\left(\frac{\mathcal{A}(r,\theta)}{r^{n-1}}\right)\right)\leq  -\int_{0}^rs\textrm{Ric}_g\left(\frac{\partial}{\partial s},\frac{\partial}{\partial s}\right)ds.$$

For $0<\epsilon<r$, we have 
$$r\log\left(\frac{\mathcal{A}(r,\theta)}{r^{n-1}}\right) -\epsilon\log\left(\frac{\mathcal{A}(\epsilon,\theta)}{\epsilon^{n-1}}\right)\leq -\int_{\epsilon}^r \int_{0}^s t\textrm{Ric}_g\left(\frac{\partial}{\partial t},\frac{\partial}{\partial t}\right)dt ds.$$

Let $\epsilon \to 0$. Then, we obtain
\begin{equation}\label{x41}
\log\left(\frac{\mathcal{A}(r,\theta)}{r^{n-1}}\right)\leq -\frac{1}{r}\int_{0}^r \int_{0}^s t\textrm{Ric}_g\left(\frac{\partial}{\partial t},\frac{\partial}{\partial t}\right)dt ds.
\end{equation}

Consider $x=(r,\theta)\in M$ and $\gamma:[0,r]\rightarrow M$ a minimizing unit geodesic with $\gamma(0)=p$ and $\gamma(r)=x$. Denote by $f(t)=f\circ \gamma (t)$, $t\in [0,r]$. Since $Ric_f\geq \rho g$, we have that
$$t\textrm{Ric}_g\left(\frac{\partial}{\partial t},\frac{\partial}{\partial t}\right) \geq t\rho -tf''(t).$$

It follows that
$$\int_0^s t\textrm{Ric}_g\left(\frac{\partial}{\partial t},\frac{\partial}{\partial t}\right)dt \geq \rho \int_0^st dt-\int_0^stf''(t) dt = \frac{\rho}{2}s^2-sf'(s)+f(s)-f(0).$$

This implies that
\begin{eqnarray*}
\int_0^r\int_0^s t \textrm{Ric}_g\left(\frac{\partial}{\partial t},\frac{\partial}{\partial t}\right) dt ds &\geq & \frac{\rho}{2}\int_0^rs^2 ds-\int_0^rsf'(s) ds+\int_0^rf(s)ds-f(0)r\\
&=&\frac{\rho}{6}r^3-rf(r)+2\int_0^rf(s) ds-f(0)r\\
&=& \frac{\rho}{6}r^3-rf(r,\theta)+2\int_0^rf(s,\theta) ds-f(p)r.
\end{eqnarray*}

By inequality (\ref{x41}), we have 
\begin{equation}\label{E}\log\left(\frac{\mathcal{A}(r,\theta)}{r^{n-1}}\right)\leq -\frac{\rho}{6}r^2+f(r,\theta)-\frac{2}{r}\int_0^rf(s,\theta) ds+f(p).\end{equation}

Since $\textrm{Ric}_f\geq \rho g$, it follows from \cite{munteanu} that there exist constants $C,r_0>0$ such that
\begin{equation}\label{EE}
f(x)\geq \frac{\rho}{2}r^2(x)-Cr(x)
\end{equation}
\noindent for every $x\in M$ with $r(x)\geq r_0$. Denote by $\lambda_0=\inf_{B_{r_0}(p)} f$. Observe that by (\ref{EE}), for $r>r_0$, we have
\begin{eqnarray*}
\int_0^r f(s,\theta)ds &=& \int_0^{r_0}f(s,\theta)ds+\int_{r_0}^r f(s,\theta)ds\\
&\geq & r_0\lambda_0+\frac{\rho}{2}\int_{r_0}^rs^2 ds-C\int_{r_0}^r s ds\\
&=& \frac{r_0}{2}\left(2\lambda_0-\frac{\rho r_0^2}{3}+Cr_0\right)+\frac{\rho}{6}r^3-\frac{C}{2}r^2
\end{eqnarray*}

It follows that, 
$$-\frac{2}{r}\int_0^r f(s,\theta)ds\leq -\frac{r_0\lambda_1}{r}-\frac{\rho}{3}r^2+Cr$$
\noindent for every $r>r_0$, where $\lambda_1:=2\lambda_0-\frac{\rho r_0^2}{3}+Cr_0$. Consequently,
\begin{equation}\label{e4}
-\frac{2}{r}\int_0^r f(s,\theta)ds\leq -\frac{\rho}{3}r^2+Cr+\lambda
\end{equation}
\noindent for every $r > r_0$, where $\lambda$ is a nonnegative constant. Substituting (\ref{e4}) into (\ref{E}) results in 
$$\log\left(\frac{\mathcal{A}(r,\theta)}{r^{n-1}}\right)\leq -\frac{\rho}{2}r^2+f(r,\theta)+Cr+\lambda+f(p).$$

This implies that
$$\mathcal{A}(r,\theta)\leq r^{n-1}e^{-\frac{\rho}{2}r^2+f(r,\theta)+Cr+\lambda+f(p)}.$$

Hence
\begin{equation}\label{e2}
\mathcal{A}(r,\theta)e^{-f(r,\theta)}\leq r^{n-1}e^{-\frac{\rho}{2}r^2+Cr+\lambda+f(p)}.
\end{equation}

Therefore, by (\ref{e2}), for every increasing function $\varphi$ and $R\geq r_0$ we have
\begin{eqnarray*}
\int_{B_R(p)} (\varphi\circ r) e^{-f}dv_g &=& \int_{B_{r_0}(p)}(\varphi\circ r) e^{-f}dv_g + \int_{B_R(p)\setminus B_{r_0}(p)} (\varphi\circ r) e^{-f}dv_g \\
&\leq & \varphi(r_0)  \int_{B_{r_0}(p)}e^{-f}dv_g +\int_{r_0}^R\int_{\mathbb{S}^{n-1}}\varphi(r)e^{-f(r,\theta)}\mathcal{A}(r,\theta)dr d\theta \\
& \leq & \varphi(r_0)\int_{B_{r_0}(p)}e^{-f}dv_g +\int_{r_0}^R\int_{\mathbb{S}^{n-1}}\varphi(r) r^{n-1} e^{-\frac{\rho}{2}r^2+Cr+\lambda+f(p)} dr d\theta \\
& = & \varphi(r_0)\int_{B_{r_0}(p)}e^{-f}dv_g +\omega_{n-1}e^{f(p)+\lambda}\int_{r_0}^R\varphi(r) r^{n-1} e^{-\frac{\rho}{2}r^2+Cr} dr  \\
& = & A\varphi(r_0)+B\int_{r_0}^R\varphi(r) r^{n-1} e^{-\frac{\rho}{2}r^2+Cr} dr, \\
\end{eqnarray*}
\noindent where 
$$A:= \int_{B_{r_0}(p)}e^{-f}dv_g >0 \ \ \ \text{and} \ \ \ B:=\omega_{n-1}e^{\lambda +f(p)}>0.$$
\end{proof}

\medskip

Another consequence of Proposition \ref{t1} is the following volume comparison result proved by Wei and Wylie in \cite{ww}.

\medskip

\begin{cor} Let $(M^n,g,f)$ be a complete smooth metric measure space such that $\textrm{Ric}_f\geq \rho g$, for some $\rho\in \mathbb{R}$. Fix a point $p\in M$. Then, there exist constants $r_0, A, B, C>0$ such that
$$\nu(B_R(p)) \leq A+B\int_{r_0}^Re^{-\frac{\rho r^2}{2}+Cr}dr$$
\noindent for every $R\geq r_0$, where $\nu=e^{-f}dv_g$. In particular, if $\rho>0$ then $\nu$ is a finite measure.
\end{cor}

\medskip

\section{An estimate for the Wasserstein distance}\label{sec4}

\medskip

In this section, we prove our main theorem (Theorem \ref{teopr}) and present some of its consequences, including Corollary \ref{copr}.

\medskip

Let $(M^n,g,f)$ be a $1$-Ricci shrinker such that $\nu=(4\pi)^{-\frac{n}{2}}e^{-f}dv_{g}$ is a probability measure and fix a point $p\in M$. The exponential map of $M$ to $p$ induces a probability measure $\overline{\nu}$ on $T_pM$ as follows. It is well-known that there is a domain $\Omega\subset T_pM$ such that $\left.\textrm{exp}_p\right|_{\Omega}:\Omega\rightarrow M\setminus \mathcal{C}(p)$ is a diffeomorphism, where $\mathcal{C}(p)$ denote the cut-locus of $p$. We define the probability measure $\overline{\nu}$ on $T_pM$ by

$$\overline{\nu}=(4\pi)^{-\frac{n}{2}}\chi_{\Omega}e^{-\overline{f}}dv_{\overline{g}}$$
\noindent where $\overline{g}$ and $\overline{f}$ denote the pullback of $g$ and $f$ by the map $\left.\textrm{exp}_p\right|_{\Omega}:\Omega\rightarrow M\setminus \mathcal{C}(p)$, respectively, i.e., $\overline{g}=\left.\textrm{exp}_p\right|_{\Omega}^*g$ and $\overline{f}=\left.\textrm{exp}_p\right|_{\Omega}^*f$.

\begin{remark} If $(M^n,g,f)$ is translated Gaussian shrinker $(\mathbb{R}^n,g_{e},\frac{|x-p|^2}{4})$ then $\overline{\nu}$ is the Gaussian measure on $\mathbb{R}^n$. 
\end{remark}

\begin{lemma}\label{l1} The measure $\overline{\nu}$ has finite second moments with respect to the Euclidean metric $g_0$.
\end{lemma}

\begin{proof} Consider $r(x)=d_g(p,x)$ the geodesic distance function to the point $p$. It is well-known that the function $r^2$ is smooth on $M\setminus \mathcal{C}(p)$ and  $r(\textrm{exp}_p(v))=|v|_{g_0}$ for every $v\in\Omega$. Then, by change of variables formula we obtain
\begin{eqnarray*}
\int_{T_pM} |v|_{g_0}^2 d\overline{\nu} &=& \int_{\Omega} |v|_{g_0}^2 (4\pi)^{-\frac{n}{2}} e^{-\overline{f}}dv_{\overline{g}}\\
&=& \int_{\Omega} r^2(\textrm{exp}_p(v)) (4\pi)^{-\frac{n}{2}} e^{-\overline{f}}dv_{\overline{g}}\\
&=& \int_{M\setminus \mathcal{C}(p)} r^2 (4\pi)^{-\frac{n}{2}}e^{-f} dv_g\\
&=& \int_{M\setminus \mathcal{C}(p)} r^2 d\nu.
\end{eqnarray*}

So, since $\mathcal{C}(p)$ is a $\nu$-null set we have
\begin{equation}\label{sm}\int_{T_pM} |v|_{g_0}^2 d\overline{\nu} =\int_{M} r^2 d\nu.\end{equation}

It follows from Theorem \ref{cz} that $\nu$ has finite second moments with respect to $g$. Hence, by equality (\ref{sm}) we have that $\overline{\nu}$ has second finite moments with respect to $g_0$.
\end{proof}

\medskip

We have two probability measures on $T_pM$ with second finite moments with respect to Euclidean metric $g_0$: the measure $\overline{\nu}$ and the Gaussian measure $\gamma$ defined in Example \ref{ex2}. Our main result provides an upper estimate for the Wasserstein distance between $\overline{\nu}$ and $\gamma$ and the rigidity arising from this estimate. Before stating the results of this section, we introduce the following notations:

\medskip
 
\begin{itemize}
\item For $s\geq 0$, we denote 
$$\Sigma_s=\{v\in T_pM; |v|_{g_0}\geq s\}.$$

\medskip

\item If $k$ is a nonnegative integer number and  $a,s\in\mathbb{R}$, we have
$$\int_s^{+\infty}r^ke^{-\frac{r^2}{4}+ar}dr< +\infty.$$

\noindent In this case, we denote
$$\Gamma(s,k,a)= \int_s^{+\infty}r^ke^{-\frac{r^2}{4}+ar}dr.$$

\medskip

\item For $k\in\mathbb{N}$ and $a,b,s\in\mathbb{R}$, we denote

$$\alpha(k,s,a,b)=(4\pi)^{-\frac{k}{2}}\omega_{k-1}e^b\left(a\Gamma(s,k,a)+b\Gamma(s,k-1,a)\right).$$

\medskip

\noindent Furthermore, we denote $\alpha(k,0,a,b)$ simply as $\alpha(k,a,b)$.
\end{itemize}

\medskip

Now let us prove our main result, that is

\medskip

\begin{theorem}[Theorem \ref{teopr}]\label{teop} The Wasserstein distance between the measure $\overline{\nu}$ on $T_pM$ and the Gaussian measure $\gamma$ on $T_pM$ satisfies
$$\frac{1}{4}W_{g_0}^2(\overline{\nu},\gamma) \leq \alpha(n,a,b) e^{f(p)-\mu_g}+f(p)-\mu_g$$

\noindent where $a$ and $b$ are nonnegative constants such that
$$f(x)\geq \frac{r^2(x)}{4}-ar(x)-b$$
\noindent for every $x\in M$. Moreover, if the equality holds then $a=b=0$ and $(M^n,g,f)$ is $(\mathbb{R}^n,g_{e},\frac{|x-p|^2}{4})$, i.e., the Gaussian shrinker up to a translation.
\end{theorem}

\medskip

Theorem \ref{teop} is a consequence of the following general result.

\medskip

\begin{theorem}\label{teog} The Wasserstein distance between the restriction of measures $\overline{\nu}$ and $\gamma$ to set $\Sigma_s$ ($s\geq 0$) satisfies
$$\frac{1}{4}W_{g_0}^2(\left.\overline{\nu}\right|_{\Sigma_s},\left.\gamma\right|_{\Sigma_s}) \leq  \alpha(n,s,a,b)\frac{e^{f(p)-\mu_g}}{\overline{\nu}(\Sigma_s)}+\log\left(\frac{\gamma(\Sigma_s)}{\overline{\nu}(\Sigma_s)}\right)+f(p)-\mu_g$$
\noindent where $a$ and $b$ are nonnegative constants such that $$f(x)\geq \frac{r^2(x)}{4}-ar(x)-b$$
\noindent for every $x\in M$ with $r(x)\geq s$. Moreover, if the equality holds then $a=b=0$ and $(M^n,g,f)$ is $(\mathbb{R}^n,g_{e},\frac{|x-p|^2}{4})$, i.e., the Gaussian shrinker up to a translation.
\end{theorem}

\begin{proof} Note that $\left.\overline{\nu}\right|_{\Sigma_s}$ is absolutely continuous with respect to $\left.\gamma\right|_{\Sigma_s}$ e
$$\frac{d\left.\overline{\nu}\right|_{\Sigma_s}}{d\left.\gamma\right|_{\Sigma_s}}=\left(\frac{\gamma(\Sigma_s)}{\overline{\nu}(\Sigma_s)}\right)\left(\frac{dv_{\overline{g}}}{dv_{g_0}}\right)\chi_{\Omega}e^{f_0-\overline{f}}.$$

It follows from Lemma \ref{l1} and Theorem \ref{ovbecz} that

\begin{eqnarray}
\frac{1}{4}W^2(\left.\overline{\nu}\right|_{\Sigma_s},\left.\gamma\right|_{\Sigma_s}) &\leq & \int_{T_pM} \log\left(\frac{d\left.\overline{\nu}\right|_{\Sigma_s}}{d\left.\gamma\right|_{\Sigma_s}}\right)d\left.\overline{\nu}\right|_{\Sigma_s} \nonumber \\
&=& \frac{1}{\overline{\nu}(\Sigma_s)}\int_{\Sigma_s} \left(f_0-\overline{f}+\log\left(\frac{dv_{\overline{g}}}{dv_{g_0}}\right)+\log\left(\frac{\gamma(\Sigma_s)}{\overline{\nu}(\Sigma_s)}\right)\right)d\overline{\nu} \nonumber \\
\label{fr0}&=& \frac{1}{\overline{\nu}(\Sigma_s)}\int_{\Sigma_s} \left(f_0-\overline{f}\right)d\overline{\nu} +\frac{1}{\overline{\nu}(\Sigma_s)}\int_{\Sigma_s}\log\left(\frac{dv_{\overline{g}}}{dv_{g_0}}\right) d\overline{\nu}\\
& & +\log\left(\frac{\gamma(\Sigma_s)}{\overline{\nu}(\Sigma_s)}\right)\nonumber
\end{eqnarray}

In terms of polar normal coordinates around $p\in M$, we may write
$$dv_g=\mathcal{A}(r,\theta)dr d\theta$$
\noindent where $d\theta$  is the area element of the unit $(n-1)$-dimensional sphere $\mathbb{S}^{n-1}$. We show in the proof of Proposition \ref{t1} that
\begin{equation}\label{x4}
\frac{d}{dr}\left(r\log\left(\frac{\mathcal{A}(r,\theta)}{r^{n-1}}\right)\right)\leq  -\int_{0}^rs\textrm{Ric}_g\left(\frac{\partial}{\partial s},\frac{\partial}{\partial s}\right)ds.\end{equation}

Let $x=(r,\theta)\in M$ and $\beta$ be a minimizing unit geodesic on $(M,g)$ such that $\beta(0)=p$ and $\beta(r)=x$. Denote by $f(s)=f(\beta(s))$, $s\in[0,r]$. Since $\textrm{Ric}_f=\frac{1}{2}g$ we have
$$\textrm{Ric}_g\left(\frac{\partial}{\partial s},\frac{\partial}{\partial s}\right)+f''(s)=\frac{1}{2}.$$

Consequently, by (\ref{x4}) we obtain

\begin{eqnarray}
\frac{d}{dr}\left(r\log\left(\frac{\mathcal{A}(r,\theta)}{r^{n-1}}\right)\right) &\leq & -\frac{r^2}{4}+\int_0^rsf''(s)ds \nonumber\\
&=& -\frac{r^2}{4}+rf'(r)-f(r)+f(0)\nonumber\\
&=& -\frac{r^2}{4}+r\langle \nabla f,\nabla r\rangle(x)-f(x)+f(p)\nonumber\\
\label{x5}&=& -\frac{r^2}{4}+r\langle \nabla f,\nabla r\rangle (x)-R_g(x)-|\nabla f|^2(x)-\mu_g+f(p)\\
&=& -\left(\frac{r}{2}-\langle \nabla f,\nabla r\rangle\right)^2(x)-R_g(x)-\mu_g+f(p)\nonumber\\
& & -(|\nabla f|^2-\langle \nabla f,\nabla r\rangle^2)(x) \nonumber\\
&\leq & -R_g(x)-\mu_g+f(p)\nonumber
\end{eqnarray}
\noindent where we have used in (\ref{x5}) the equality $R_g+|\nabla f|^2=f-\mu_g$. It follows that
$$\frac{d}{dr}\left(r\log\left(\frac{\mathcal{A}(r,\theta)}{r^{n-1}}\right)\right) \leq -R_g(r,\theta)-\mu_g+f(p).$$

So, for $0<\epsilon<r$ we have

$$r\log\left(\frac{\mathcal{A}(r,\theta)}{r^{n-1}}\right) -\epsilon\log\left(\frac{\mathcal{A}(\epsilon,\theta)}{\epsilon^{n-1}}\right) \leq  -\int_{\epsilon}^r R_g(s,\theta)ds+(f(p)-\mu_g)(r-\epsilon).$$

Let $\epsilon \to 0$. Then, we obtain
$$\log\left(\frac{\mathcal{A}(r,\theta)}{r^{n-1}}\right)\leq  -\frac{1}{r}\int_{0}^r R_g(s,\theta)ds+f(p)-\mu_g.$$

This implies that
$$\mathcal{A}(r,\theta)\leq r^{n-1}e^{-\frac{1}{r}\int_{0}^r R_g(s,\theta)ds+f(p)-\mu_g}.$$

So, since $R_g\geq 0$ we have 
\begin{equation}\label{fr}
\mathcal{A}(r,\theta)\leq r^{n-1} e^{f(p)-\mu_g}.
\end{equation}

In terms of polar normal coordinates around the origin $o\in T_pM$, we have
$$dv_{\overline{g}}=\mathcal{A}(r,\theta)dr d\theta \ \ \ \text{and} \ \ \ dv_{g_0}=r^{n-1} dr d\theta .$$

This implies that $$\left(\frac{dv_{\overline{g}}}{dv_{g_0}}\right)(r,\theta)=\frac{\mathcal{A}(r,\theta)}{r^{n-1}}.$$

By inequality (\ref{fr}) we obtain $$\frac{dv_{\overline{g}}}{dv_{g_0}} \leq e^{f(p)-\mu_g}.$$

This implies that
\begin{equation}\label{fr2} \frac{1}{\overline{\nu}(\Sigma_s)}\int_{\Sigma_s}\log\left(\frac{dv_{\overline{g}}}{dv_{g_0}}\right) d\overline{\nu}\leq f(p)-\mu_g.\end{equation}

Since $\mathcal{C}(p)$ is a $\nu$-null set, the function $r^2$ is smooth on $M\setminus \mathcal{C}(p)$ and $|v|_{g_0}=r(\textrm{exp}_p(v))$ for every $v\in \Omega$, we have
$$\int_{\Sigma_s} \left(f_0-\overline{f}\right)d\overline{\nu} = \int_{\Sigma'_s}\left(\frac{r^2}{4}-f\right)d\nu$$
\noindent where $\Sigma'_s:=\{x\in M; r(x)\geq s\}.$ It follows that
\begin{eqnarray}
\int_{\Sigma_s} \left(f_0-\overline{f}\right)d\overline{\nu} &=& \int_{\Sigma'_s}\left(\frac{r^2}{4}-f\right)(4\pi)^{-\frac{n}{2}}e^{-f}dv_g \nonumber\\
&=& (4\pi)^{-\frac{n}{2}}\int_s^{+\infty}\int_{\mathbb{S}^{n-1}}\left(\frac{r^2}{4}-f(r,\theta)\right)e^{-f(r,\theta)}\mathcal{A}(r,\theta)d\theta dr \nonumber \\
\label{fr3}&\leq & (4\pi)^{-\frac{n}{2}}e^{f(p)-\mu_g}\int_s^{+\infty}\int_{\mathbb{S}^{n-1}}\left(\frac{r^2}{4}-f(r,\theta)\right)r^{n-1}e^{-f(r,\theta)}d\theta dr\\
\label{fr4}&\leq & (4\pi)^{-\frac{n}{2}}e^{b+f(p)-\mu_g}\int_s^{+\infty}\int_{\mathbb{S}^{n-1}}\left(ar+b\right)r^{n-1}e^{-\frac{r^2}{4}+ar}d\theta dr\\
& = & (4\pi)^{-\frac{n}{2}}\omega_{n-1}e^{b+f(p)-\mu_g}\int_s^{+\infty}\left(ar+b\right)r^{n-1}e^{-\frac{r^2}{4}+ar} dr \nonumber\\
& = & (4\pi)^{-\frac{n}{2}}\omega_{n-1} e^{b+f(p)-\mu_g}\left(a\Gamma(s,n,a)+b\Gamma(s,n-1,a)\right)\nonumber
\end{eqnarray}
\noindent where we have used in (\ref{fr4}) we have used that $f(x)\geq\frac{r^2(x)}{4}-ar(x)-b$ holds for every $x\in M$ with $r(x)\geq s$ and in (\ref{fr3}) we have used the inequality (\ref{fr}). This implies that
\begin{equation}\label{fr5}
\int_{\Sigma_s} \left(f_0-\overline{f}\right)d\overline{\nu}\leq \alpha(n,s,a,b) e^{f(p)-\mu_g}
\end{equation}

Hence, by (\ref{fr0}), (\ref{fr2}) and (\ref{fr5}) we obtain
\begin{equation}\label{xz}\frac{1}{4}W_{g_0}^2(\left.\overline{\nu}\right|_{\Sigma_s},\left.\gamma\right|_{\Sigma_s}) \leq  \alpha(n,s,a,b)\frac{e^{f(p)-\mu_g}}{\overline{\nu}(\Sigma_s)} +\log\left(\frac{\gamma(\Sigma_s)}{\overline{\nu}(\Sigma_s)}\right)+f(p)-\mu_g.\end{equation}

Assume that the equality holds in (\ref{xz}). This implies that the equality in (\ref{fr3}) and (\ref{fr4}) holds. Consequently,

\medskip

\begin{enumerate}

\item The equality in (\ref{fr}) holds for every $r\geq s$, i.e., $\mathcal{A}(r,\theta)=r^{n-1}e^{f(p)-\mu_g},$ for every $r\geq s$;

\medskip

\item $f(x)=\frac{r^2(x)}{4}-ar(x)-b$, for every $x\in\mathbb{R}^n$ with $r(x)\geq s$.
\end{enumerate}

\medskip

By (1) we have $$\int_0^rR_g(t,\theta)dt=0,$$\noindent for every $r\geq s$. Since $R_g\geq 0$ we obtain $R_g\equiv 0$. It follows that $(M,g)$ is isometric to $(\mathbb{R}^n,g_{e})$. By passing an isometry, we have that there is $y\in\mathbb{R}^n$ such that 
$$f(x)=\frac{|x-p|^2}{4}+\langle x-p, y\rangle+ |y|^2+\mu_g$$
\noindent for every $x\in \mathbb{R}^n$. By (2) we obtain that $a=0$, $y=o$ and $b=-\mu_g$. This implies that $f(x)=\frac{|x-p|^2}{4}+\mu_g$, for every $x\in\mathbb{R}^n$. Since $\nu$ is a probability measure we have $\mu_g=0$. Therefore, $(M^n,g,f)$ is $(\mathbb{R}^n,g_{e},\frac{|x-p|^2}{4})$ and $a=b=0$.
\end{proof}

By Remark \ref{rentropy}, we have that if $f$ attains a global minimum at $p$ then $$f(p)-\mu_g=R_g(p).$$

So, we conclude this paper with the following corollaries of Theorem \ref{teog}.

\begin{cor} If $f$ attains a global minimum at $p$, then the Wasserstein distance between the restriction of measures $\overline{\nu}$ and $\gamma$ to set $\Sigma_s$ ($s\geq 0$) satisfies
$$\frac{1}{4}W_{g_0}^2(\left.\overline{\nu}\right|_{\Sigma_s},\left.\gamma\right|_{\Sigma_s}) \leq  \alpha(n,s,a,b)\frac{e^{R_g(p)}}{\overline{\nu}(\Sigma_s)} +\log\left(\frac{\gamma(\Sigma_s)}{\overline{\nu}(\Sigma_s)}\right)+R_g(p)$$
\noindent where $a$ and $b$ are nonnegative constants such that $$f(x)\geq \frac{r^2(x)}{4}-ar(x)-b$$
\noindent for every $x\in M$ with $r(x)\geq s$. Moreover, if the equality holds then $a=b=0$ and $(M^n,g,f)$ is $(\mathbb{R}^n,g_{e},\frac{|x-p|^2}{4})$, i.e., the Gaussian shrinker up to a translation.
\end{cor}

\begin{cor}[Corollary \ref{copr}] If $f$ attains a global minimum at $p$ then the Wasserstein distance between the measure $\overline{\nu}$ on $T_pM$ and the Gaussian measure $\gamma$ on $T_pM$ satisfies
$$\frac{1}{4}W_{g_0}^2(\overline{\nu},\gamma) \leq \alpha(n,a,b)e^{R_g(p)}+R_g(p)$$
\noindent where $a$ and $b$ are nonnegative constants such that
$$f(x)\geq \frac{r^2(x)}{4}-ar(x)-b$$
\noindent for every $x\in M$. Moreover, if the equality holds then $a=b=0$ and $(M^n,g,f)$ is $(\mathbb{R}^n,g_{e},\frac{|x-p|^2}{4})$, i.e., the Gaussian shrinker up to a translation.
\end{cor}

\end{document}